\documentclass[runningheads,envcountsame]{llncs}
\usepackage[T1]{fontenc}
\usepackage{amssymb, amsmath}
\usepackage{xcolor}
\usepackage[style = numeric, backend = biber, alldates = year, maxbibnames=9, maxcitenames=4, doi=false, isbn=false, giveninits = true]{biblatex}
\usepackage{colonequals}

\usepackage{hyperref}
\hypersetup{pdftitle={Belyi map verification using certified path tracking}}
\hypersetup{pdfauthor={Alexandre Guillemot and John Voight}}
\hypersetup{colorlinks=true,linkcolor=blue,anchorcolor=blue,citecolor=blue}

\newcommand{\CC}{\mathbb{C}}
\newcommand{\RR}{\mathbb{R}}
\newcommand{\QQ}{\mathbb{Q}}
\newcommand{\PP}{\mathbb{P}}
\newcommand{\C}{\CC}

\newcommand{\QQbar}{\QQ^{\textup{al}}}

\newcommand{\defi}{\textsf}

\makeatletter
\let\c@equation\c@theorem

\makeatother

\usepackage{xpatch}
\xpatchbibmacro{textcite}
  {\printnames{labelname}}
  {\begingroup
     \printnames{labelname}%
   \endgroup}
  {}{}

\begin{document}
\title{Belyi map verification using certified path tracking}
\author{Alexandre Guillemot\inst{1} \and John Voight\inst{2}}
\authorrunning{A.\ Guillemot \and J.\ Voight}
\institute{Inria, Universit\'e Paris--Saclay, Palaiseau, 91120, France\\
  \email{alexandre.guillemot@inria.fr}\\
  \and School of Mathematics and Statistics, University of Sydney, NSW, 2006, Australia\\
\email{jvoight@gmail.com}}

\maketitle

\begin{abstract}
We provide an end-to-end workflow to rigorously compute the monodromy of Belyi maps from exact equations over number fields using certified homotopy continuation.  We then apply this method at scale to certify the monodromy triples of Belyi maps in the $L$-functions and Modular Forms Database (LMFDB).
  
\end{abstract}

\section{Introduction}

\subsection*{Motivation and context}

Let $X$ be a smooth projective algebraic curve over $\C$ or equivalently a compact Riemann surface.  
A \defi{Belyi map} is a nonconstant map $\varphi\colon X \to\PP^1_{\CC}$ unramified away from~$\{0, 1, \infty\}$.  \textcite{belyiGALOISEXTENSIONSMAXIMAL1980} proved that $X$ can be defined over~$\QQbar$ if and only if it admits a Belyi map.  Explicit methods for Belyi maps have seen many applications---see \textcite{sijslingComputingBelyiMaps2015} for an overview, as well as the more recent work by \textcite{robertsHurwitzBelyiMaps2018,barthComputationBelyiMaps2021,MR4709206}.  These methods take as input a permutation triple encoding the monodromy around the branch points and return candidate equations over a number field $K \subset \C$.  One can certify that the map is indeed a Belyi map with the correct ramification.

In this paper, we tackle the certification that the monodromy triple of the map indeed coincides with the given input. 
Algebraic techniques \cite{elkiesComplexPolynomialsGal2013,barthComputationBelyiMaps2021} may suffice to compute the monodromy group and the triple when uniquely determined, but they do not provide a general approach.
On the other hand, monodromy permutations can be computed from the equations using \emph{certified path tracking}: for a parametrized polynomial system $F_t$, one follows the zeros of $F_t$ along loops in the base.  For monodromy, this requires isolating the tracked zero along the whole path to ensure that path jumping or swapping does not occur \cite[\S 6.2]{beltranCertifiedNumericalHomotopy2012}, ensuring correctness.

\subsection*{Previous work}

The idea of using certified path tracking or homotopy continuation to rigorously compute the monodromy of Belyi maps is not new.
For instance, \textcite{schnepsDessinsDenfantsRiemann1994} and \textcite[][\S 2C]{MR3952010} describe predictor-corrector loops whose certification rely on a step size bound.
\textcite{deconinckComputingRiemannMatrices2001} and \textcite{bartholdiAlgorithmicConstructionHurwitz2015} present approaches that compute fibers above the points of a discretization of the parameter path and connect each consecutive fibers to recover the monodromy permutation.
However, implementing these methods to obtain a rigorous result is difficult, as they may consider implicit bounds, that roots of polynomials can be obtained exactly, or assume a model of exact computation over the reals (e.g., the BSS model) for correctness.

\textcite{MR3032682} provide the first implemented certified homotopy continuation algorithm in the Turing machine model, using exact arithmetic over the rationals and with generality well beyond Belyi maps.
An \emph{a posteriori} certification algorithm was designed by \textcite{10.1145/2608628.2608651} for Newton homotopies (those that change only the constant term), taking a heuristic tracking scheme as input to produce a certificate that no path jumpings occurred; the authors consider monodromy computations as an application.
Other methods rely on interval arithmetic, such as \textcite{kearfottIntervalStepControl1994}, as well as \textcite{guillemotValidatedNumericsAlgebraic2024,duffCertifiedHomotopyTracking2024}, based on ideas from \textcite{vanderhoevenReliableHomotopyContinuation2015}.
The preprint by \textcite{duffCertifyingGaloisMonodromy2026} addresses monodromy computations for general complex algebraic branched covers, with an emphasis on practical results and including Belyi maps.
In the univariate case, certified path tracking methods are more diverse; see the algorithms and implementations by
\textcite{marco-buzunarizSIROCCOLibraryCertified2016,kranichEpsilondeltaBoundPlane2016,xuApproachCertifyingHomotopy2018}.

\subsection*{Contribution} Using Algpath \cite{guillemotCertifiedAlgebraicPath2026,guillemotValidatedNumericsAlgebraic2024}, a certified path tracking software based on Krawczyk's operator, interval arithmetic, and Taylor models, we rigorously compute the monodromy triples of the $1111$ Belyi maps present in the LMFDB \cite{mustyDatabaseBelyiMaps2019}.  Code is available online \cite{ourcode}.  

Our contribution is not a new path-tracking algorithm, but rather we develop a complete and scalable certification pipeline for exact Belyi maps over number fields and run this on the complete set of Belyi maps in the LMFDB.  This demonstrates that recent advances make certified homotopy continuation a practical tool to compute the monodromy of finite algebraic maps.  

\section{Setup}

\subsection*{Monodromy representation}

Let $X$ be a smooth projective algebraic curve over $\CC$ and let $\varphi \colon X \to \PP^1_\CC$ be a Belyi map of degree $d$.  Put $U \colonequals \PP^1_\CC\setminus\{0,1,\infty\}$, so that $\varphi$ restricts to a topological cover $\varphi^{-1}(U)\to U$ of degree $d$.

Fix a basepoint $b \in U$ and label the fiber $\varphi^{-1}(b)=\{x_1,\dots,x_d\}$.  Analytic continuation along loops in $U$ defines a monodromy representation $\rho \colon \pi_1(U,b)\to S_d$.  Choosing standard loops $\gamma_0,\gamma_1,\gamma_\infty$ around $0,1,\infty$ with $\gamma_0\gamma_1\gamma_\infty=1$, we obtain permutations $\sigma_0,\sigma_1,\sigma_\infty$ satisfying $\sigma_0\sigma_1\sigma_\infty=1$, called the \defi{monodromy triple} of $\varphi$, well-defined up to simultaneous conjugation.  The subgroup $G=\langle \sigma_0,\sigma_1,\sigma_\infty\rangle\subseteq S_d$ is the \defi{(geometric) monodromy group}.  The cycle structure of $\sigma_s$ records the ramification above $s \in \{0,1,\infty\}$.

\subsection*{Analytic continuation via polynomial systems}

To compute this triple from equations, let $Y \subseteq \C^n$ be an affine chart cut out by polynomials $g_1,\dots,g_{n-1}$, and suppose $\varphi$ is represented on this chart by a rational function $p/q$.  Consider the polynomial system
\begin{equation}\label{eq:theory-system}
  F_t(x,z)=\bigl(g_1(x),\dots,g_{n-1}(x),\,p(x)-tq(x),\,q(x)z-1\bigr).
\end{equation}
For each $t \in U$, the regular zeros of $F_t$ are in bijection with the points of $\varphi^{-1}(t)$, since $q(x)z-1=0$ excludes the common vanishing locus of $p$ and $q$.  Over $U$, these zeros are regular and vary analytically with~$t$, so analytic continuation of zeros of~\eqref{eq:theory-system} agrees with path lifting for $\varphi$.

Therefore, if $\gamma\colon [0, 1] \to U$ is a loop based at $b$ and each regular zero of~$F_{b}$ is tracked correctly along~$\gamma$, then the induced permutation on the zeros of~$F_{b}$ is the monodromy permutation of~$\varphi$ associated to~$\gamma$.  In particular, it is enough to compute the permutations associated to loops around $0$ and~$1$. 

\section{Monodromy computations using path tracking}
\label{sec:monodromy-path-tracking}
We briefly explain the underlying algorithm of Algpath and how it applies to the computation of monodromy permutations.

\subsection*{Data structure for regular zeros of polynomial systems}
Let~$f\colon \CC^n \to \CC^n$ be a polynomial map.  We review Moore's root isolation criterion, based on Krawczyk's operator, and the associated data structure for regular zeros.  We denote by~$B$ the unit ball for the real~$\infty$-norm on~$\CC^n$, and by $I_n$ the $n \times n$ identity matrix.
\begin{theorem}[Moore's criterion]\label{thm:moore-crit}
  Let~$\rho \in (0, 1)$,~$x \in \CC^n$,~$r > 0$ and~$A \in \CC^{n \times n}$.
  Suppose that for all~$u, v \in rB$, we have
  \begin{equation}\label{cond:moore}
    -Af(x) + (I_n - A \,\mathrm{d}f(x + u))v \in \rho r B.
  \end{equation}
  Then there exists a regular zero~$\zeta$ of~$f$ in~$x + \rho r B$, and it is the unique zero of~$f$ in $x + rB$.
\end{theorem}

\begin{proof}
See \cite{rumpSOLVINGALGEBRAICPROBLEMS1983,guillemotValidatedNumericsAlgebraic2024}.
\end{proof}

\begin{definition}[Moore boxes]
  Let~$\rho \in (0, 1)$.  A~\defi{$\rho$-Moore box} for~$f$ is a triple
  $(x, r, A) \in \CC^n \times \RR_{>0} \times \CC^{n \times n}$ satisfying~\eqref{cond:moore}.
\end{definition}
A \defi{Moore box} is a~$\rho$-Moore box for some~$\rho \in (0, 1)$.  By Theorem~\ref{thm:moore-crit}, a Moore box represents a unique regular zero of~$f$.  Condition~\eqref{cond:moore} is checked effectively by interval arithmetic~\cite{mooreIntroductionIntervalAnalysis2009}.  Given a~$\rho$-Moore box~$m=(x,r,A)$ and~$\tau \in (0,1)$, one can refine~$m$ to a~$\tau$-Moore box with the same associated zero by quasi-Newton iterations and reductions of~$r$ \cite[Algorithm~2]{guillemotValidatedNumericsAlgebraic2024}.  Equality of the zeros represented by two Moore boxes is decided by refining both to~$1/3$-Moore boxes and checking whether the shrunken boxes intersect.

\subsection*{Path tracking algorithm}
We focus on the algorithm implemented in Algpath; a precise description is given in~\cite{guillemotValidatedNumericsAlgebraic2024}.  Let~$F\colon \CC \times \CC^n \to \CC^n$ be a polynomial map.
The first argument is the parameter, put in subscript, so that~$F_t\colon \CC^n \to \CC^n$ is the map obtained from~$F$ by specialization of the parameter.
Let~$x \in \CC^n$ be a regular zero of~$F_0$, and let~$\zeta\colon [0, 1] \to \CC^n$ be the corresponding solution path.  The input zero is given by a Moore box~$m$, and Algpath alternates:
\begin{enumerate}
  \item \defi{correction}: refine~$m$ to a~$1/8$-Moore box;
  \item \defi{prediction}: compute $\delta > 0$ such that for all~$s \in [t,t+\delta]$,~$m$ is a~$7/8$-Moore box for~$F_s$, then update~$t$ to~$t+\delta$.
\end{enumerate}
When~$t=1$, the algorithm returns a Moore box for~$\zeta(1)$.  \textcite{guillemotValidatedNumericsAlgebraic2024} prove termination and correctness in a practical computational model, and explain how predictors and Taylor models improve efficiency.

To track a zero of~$F$ along a more general path~$\gamma \colon [0, 1] \to \CC$, we precompose~$F$ with~$\gamma$ in the~$t$ variable.  In this article, the parameter paths are piecewise linear loops, and we track each linear piece in turn.

\subsection*{Computing monodromy permutations}

Let~$F\colon \CC \times \CC^n \to \CC^n$ be a polynomial map, let~$X$ be the variety of pairs~$(t, x)$ such that~$F(t, x) = 0$, let~$p$ be the projection on the first variable, and denote by~$X_t = p^{-1}(t)$ the fiber above~$t \in \CC$.  
Let~$U \subseteq \CC$ over which $X_t$ contains finitely many regular points and has constant cardinality, so that $p$~induces a cover over~$U$.

Let~$\gamma\colon [0, 1] \to U$ be a loop based at~$b \in U$.  To compute the monodromy permutation associated to~$\gamma$, we first compute the fiber~$X_b$ and represent its points by Moore boxes.  We then run certified path tracking along~$\gamma$ from each point of the fiber.  The resulting ending boxes again represent the points of~$X_b$, and the monodromy permutation is obtained by identifying which starting and ending Moore boxes represent the same zero.

\subsection*{Implementation details}

Algpath is a Rust software package.
It provides a command-line interface (CLI) that accepts as input a JSON file describing a parametrized polynomial map~$F$, the parameter path, and a list of starting points.
For detailed usage documentation, see \cite{algpath}.

From the list of starting zeros given as input, Algpath tracks the corresponding solution paths independently and in parallel.
When computing a monodromy permutation on~$d$ points (using the \texttt{-{}-monodromy} CLI option), the program halts whenever~$d - 1$ path tracking processes finish, as a permutation is uniquely determined by its action on~$d-1$ points.
This improves computational time when one zero is slower to track.
For instance, consider the Belyi map with LMFDB entry \href{https://www.lmfdb.org/Belyi/7T7/6.1/5.2/4.2.1/a/}{\textsf{7T7-6.1\_5.2\_4.2.1-a}}, using the second embedding provided by the database.
When computing~$\sigma_0$ with one thread, one solution path accounts for 95\% of the total tracking time.
When Algpath is run with six threads, the early-termination optimization yields a factor-of-50 speedup over tracking all solution paths on this example.

The software relies on two interval arithmetics.
The first is a SIMD double-precision interval arithmetic implemented within Algpath, following \textcite{lambovIntervalArithmeticUsing2008}; we refer to it as SIMD interval arithmetic (currently, the software supports only x86-64 CPUs with AVX2).
The second is Arb \parencite{johanssonArbEfficientArbitraryPrecision2017}.
Although arbitrary-precision, as provided by Arb, is required for Algpath to terminate, SIMD interval arithmetic performs better when double precision suffices: for instance, using Arb to tackle a problem solvable in double precision is roughly 30--40 times slower than using SIMD interval arithmetic.
For this reason, the software switches between the two arithmetics depending on the precision required for the computation.
For more details, see \cite{guillemotCertifiedAlgebraicPath2026}.

Other key features include certified prediction along Hermite's predictor using Taylor models (see \cite{guillemotValidatedNumericsAlgebraic2024}), projective path tracking and straight-line program representation of~$F$.

\section{Verification of Belyi maps}
The method of~\textcite{klugNumericalCalculationThreepoint2014}, used for the LMFDB \cite{mustyDatabaseBelyiMaps2019}, computes equations for the Belyi map associated to a given permutation triple~$\sigma$.  Once the equations are computed, it remains to check that the monodromy triple of the equations is simultaneously conjugate to~$\sigma$.  Using Algpath, we perform this check for all Belyi maps present on the LMFDB.

We use the algorithm and Moore-box machinery described in Section 3, with the usual reduction of a finite map to a parametrized polynomial system.  In this section, we elaborate upon additional techniques, also found in other places in the literature, that improve performance: introducing $qz-1$ to remove the base locus, certifying a complete starting fiber by root isolation and a degree count, and representing a chosen complex embedding of a number field by interval coefficients.  Our contribution is to assemble these ingredients into an end-to-end rigorous workflow for the Belyi map models occurring in the LMFDB.  The principal new result is the certification of the monodromy data for all entries in the database.

\subsection*{Parametric system building}
We explain how to build a parametric system to compute the monodromy of a given Belyi map for each data type encountered on the LMFDB.

\subsubsection{Smooth model.}
In this case, the Belyi map is given by a polynomial~$f \in \CC[x, y]$ and a bivariate rational function~$p/q \in \CC(x, y)$.  By homogenizing, we obtain a smooth curve~$X \subseteq \PP^2$, and~$p/q$ induces a map~$\varphi\colon X \to \PP^1$.  The system we use to compute the monodromy is
\begin{equation}\label{eq:smooth-system}
    f = p-tq = 0.
\end{equation}
For fixed~$t \in \CC$, this system captures the fiber~$\varphi^{-1}(t)$ together with the points of~$X$ on which~$p$ and~$q$ vanish simultaneously.  Adding~$qz - 1 = 0$ removes the latter but may greatly slow down path tracking.  In practice, we compute starting solutions using the enlarged system and then track~\eqref{eq:smooth-system}, excluding the solutions on which~$q$ vanishes.  More precisely, from a Moore box for the enlarged system with center~$(x,y,z)$ we compute a Moore box for~\eqref{eq:smooth-system} centered at~$(x,y)$; Theorem~\ref{thm:moore-crit} ensures that~$q$ does not vanish on the associated zero.

\subsubsection{Smooth case with curve~$\PP^1$.}
Whenever the underlying curve is~$\PP^1$, the LMFDB provides a rational function~$p/q \in \CC(x)$ representing a map~$\PP^1 \to \PP^1$.  In this case, we compute the monodromy using~$p - tq = 0$ after checking that~$p$ and~$q$ are coprime.

\subsubsection{Plane models.}
Some entries provide, in addition to the smooth model, one polynomial equation~$f \in \CC[t, x]$ defining a possibly singular curve, and a constant~$\lambda \in \CC$.  If~$p \colon V(f) \to \CC$ is the projection on the variable~$t$, then~$\lambda p$ induces a map from the normalization of~$V(f)$ to~$\PP^1$, and we compute the monodromy by tracking the roots of~$f(x,\lambda^{-1}t)$.

\subsection*{Dealing with algebraic coefficients}
Algpath can handle polynomial systems with rational, floating point, or interval coefficients, whereas the equations defining Belyi maps may have algebraic coefficients.  Each LMFDB entry specifies a number field~$\QQ(\nu)$ by the minimal polynomial~$m_\nu$ of~$\nu$, and the coefficients of the equations are polynomials in~$\nu$ with rational coefficients.  Given an embedding~$\QQ(\nu)\hookrightarrow\CC$, corresponding to a root~$\alpha$ of~$m_\nu$, we obtain a Belyi map by evaluating at~$\alpha$.

One possibility is to treat~$\nu$ as an additional variable, add~$m_\nu$ to the system, and extend each starting zero with~$\alpha$ on the~$\nu$-coordinate.  This ensures termination but introduces a substantial cost, and the closeness of the roots of~$m_\nu$ may force small Moore boxes and hence small step sizes.

In practice, we instead compute a tight interval enclosure of~$\alpha$ and replace~$\nu$ by this interval in the system.  If the interval is not sufficiently tight, the algorithm may stall, but whenever a run succeeds, the result is correct.
Moreover, if we initially compute (an interval enclosing)~$\alpha$ with more precision than the precision required to track a zero, the software will terminate.
Although the precision required to track a zero cannot be known in advance, we observe experimentally that it usually remains below 200 bits.
Thus, we compute~$\alpha$ with~$10^4$ bits of precision; this is sufficient to prevent stalling in all the examples considered in this article, while this computation is inexpensive in \textsf{SageMath} and does impact the performance of Algpath, where costly computations happen at low precision.

This performs much better in practice: for the Belyi map with entry \href{https://www.lmfdb.org/Belyi/7T7/6.1/5.2/4.2.1/a/}{\textsf{7T7-6.1\_5.2\_4.2.1-a}}, adding the number field equation did not finish within~$6$ days, whereas replacing~$\alpha$ by an interval took less than~$20$ seconds.

\subsection*{Computing the fiber above the base point}
To compute the monodromy of a Belyi map, we first need a Moore box for each point in the fiber above the base point for the loops around~$0$ and~$1$.  We numerically compute the fiber and then certify~$1/3$-Moore boxes around the resulting approximations.  The whole fiber is captured as soon as the number of Moore boxes matches the degree of the map and they represent distinct zeros, which we check using the equality test of Section~\ref{sec:monodromy-path-tracking}.

In the smooth case where the curve is~$\PP^1$, or when a plane model is available, this amounts to solving a univariate polynomial with complex coefficients after embedding the equations into~$\CC$, so we use a standard root finder in \textsf{SageMath}.  Otherwise we use \texttt{Msolve} \cite{berthomieuMsolveLibrarySolving2021} on~\eqref{eq:smooth-system}, before embedding the equations into~$\CC$ and treating~$\nu$ as a variable.  This yields fibers, for all embeddings of the number field, which we sort using the~$\nu$-coordinate.  For each fiber, we also get the unwanted common locus of~$p$ and~$q$ over~$X$; to remove it, we discard the points whose evaluation by~$q$ is the smallest, until the size of the fiber matches the degree of the map.  We choose \texttt{Msolve} as it is directly accessible from \textsf{SageMath} and requires no additional setup, but we also experimented with \texttt{HomotopyContinuation.jl}~\cite{breidingHomotopyContinuationjlPackageHomotopy2018} and found it to be a suitable alternative for this task.

\subsection*{Practical results}
We compute the triples for all~$1111$ entries of the LMFDB database.
The certification code is written in \textsf{SageMath}; it computes the monodromy of a given Belyi map using Algpath and processes the output.
The results (triples, monodromy computation time using Algpath and verifications against the LMFDB for each entry) and the \textsf{SageMath} scripts can be found online \cite{ourcode}.  Using plane models when available, the whole computation takes~$1.2$ hours of CPU time and the total wall clock time is~$17$ minutes on~$6$ threads using an Intel Core i7-1370P CPU.

Using the resulting triples, we prove correctness of all monodromy groups in the database.  The computations brought to light some errors in the data.  First, the triples computed using plane models only correspond to the triples on the database up to an~$S_3$-action, coming from post-composition by an automorphism of~$\PP^1$ that permutes~$\{0,1,\infty\}$.  In this case, we check correctness only up to this~$S_3$-action, which can be recovered from the smooth models provided.  Second, there were 7 passports with data errors: in 5 cases, embeddings are in the wrong order and in the other 2 cases there seems to be duplicate data.  These issues have been reported.  Other than that, all remaining triples were proved correct.

Additionally, we verified the monodromy of the maps provided by \textcite{barthComputationBelyiMaps2021}, whose degree ranges from~$55$ to~$280$, with the exception of two: the two maps with highest degree ($266$ and~$280$) are defined over a non-trivial number field, and while their monodromy group could be verified using algebraic techniques, their monodromy triples remain unverified.

\section{Conclusion}

We have presented a certified method, implemented in Algpath, for recovering the monodromy of a Belyi map from exact equations over number fields.  The method tracks regular solutions of a parametrized polynomial system along loops using interval arithmetic, Krawczyk-based root isolation, and certified homotopy continuation.  Finally, we showed that this approach works at scale, certifying the monodromy triples of all Belyi maps currently in the LMFDB. 

\subsection*{Acknowledgements}
We are grateful to Anton Leykin, Pierre Lairez, and \'Eric Pichon-Pharabod for useful discussions, as well as to Joshua Perlmutter for his work at initial stages of the project.  
Guillemot was supported by the European Research Council (ERC) under the European Union's Horizon Europe research and innovation program, grant agreement 101040794 (10000 DIGITS).
Voight was supported by a grant from the Simons Foundation (SFI-MPS-Infra\-structure-00008650).

\renewcommand*{\bibfont}{\small}
\printbibliography
\end{document}